\documentclass[preprint,12pt]{elsarticle}
\usepackage{amsthm,amsfonts,amssymb,amscd,amsmath,enumerate,verbatim,calc,graphicx,geometry}
\usepackage[all]{xy}
\newtheorem{theorem}{Theorem}[section]
\newtheorem{lemma}[theorem]{Lemma}
\newtheorem{proposition}[theorem]{Proposition}
\newtheorem{corollary}[theorem]{Corollary}
\theoremstyle{definition}
\theoremstyle{definitions}
\newtheorem{definition}[theorem]{Definition}

\newtheorem{remark}[theorem]{Remark}
\newtheorem{example}[theorem]{Example}
\theoremstyle{notations}

\theoremstyle{remarks}

\journal{ }

\begin{document}

\begin{frontmatter}



\title{On the Capacity of Eilenberg-MacLane and Moore Spaces}


\author[]{Mojtaba~Mohareri}
\ead{m.mohareri@stu.um.ac.ir}
\author[]{Behrooz~Mashayekhy\corref{cor1}}
\ead{bmashf@um.ac.ir}
\author[]{Hanieh~Mirebrahimi}
\ead{h$_{-}$mirebrahimi@um.ac.ir}
\address{Department of Pure Mathematics, Center of Excellence in Analysis on Algebraic Structures, Ferdowsi University of
Mashhad,\\
P.O.Box 1159-91775, Mashhad, Iran.}
\cortext[cor1]{Corresponding author}
\begin{abstract}
 K. Borsuk in 1979, in the Topological Conference in Moscow, introduced the concept of the capacity of a compactum and asked some questions concerning properties of the capacity of
compacta. In this paper, we give partial positive answers to three of  these questions in some cases. In fact, by describing spaces homotopy dominated by Moore and Eilenberg-MacLane spaces, we obtain the capacity of a Moore space $M(A,n)$ and an Eilenberg-MacLane space $K(G,n)$. Also, we compute the capacity of the wedge sum of finitely many Moore spaces of  different degrees and the capacity of the product of finitely many Eilenberg-MacLane spaces of  different homotopy types. In particular, we give exact capacity of the wedge sum of finitely many spheres of the same or different dimensions.
\end{abstract}


\end{frontmatter}


\section{Introduction and Motivation}
 K. Borsuk in \cite{So}, introduced the concept of capacity of a compactum (compact metric space) as follows:
 the capacity $C(A)$ of a compactum $A$ is the cardinality of the set of all shapes of compacta $X$ for which $\mathcal{S}
h(X) \leqslant \mathcal{S}h(A)$. Similarly, we can define the capacity for any topological space $A$ as the cardinality of the set of all shapes of spaces $X$ for which $\mathcal{S}
h(X) \leqslant \mathcal{S}h(A)$.

In the case polyhedra,  the notions shape and shape domination in the above definition can be replaced by the notions homotopy type and homotopy domination, respectively. Indeed, by some known
results in shape theory we conclude that for any polyhedron $P$, there is a 1-1 functorial correspondence between the shapes of compacta shape dominated by $P$ and the homotopy types of CW-complexes (not necessarily finite) homotopy dominated by $P$ (in both pointed and unpointed cases) \cite{18}.

 It is obvious that the capacity of a topological space is a homotopy invariant, i.e., if topological spaces $X$ and $Y$ have the same homotopy type, then $C(X)=C(Y)$. Now, it is interesting to know that what topological spaces have finite capacity. Of course, S. Marther in \cite{17} proved that every polyhedron dominates only a countable number of different homotopy types (hence shapes).

In addition, Borsuk in \cite{So} asked a question: `` Is it true that the capacity of every finite polyhedron is finite? ''.  D. Kolodziejczyk in \cite{16} gave a negative answer to this question. However, she investigated some conditions for polyhedra to have finite capacity (\cite{13, 18, 14, 15}). For instance, polyhedra with finite fundamental groups and polyhdera $P$ with abelian fundamental groups $\pi_1 (P)$ and finitely generated homology groups $H_i (\tilde{P})$, for $i\geq 2$, have finite capacity.

Also, Kolodziejczyk has studied on the capacity of CW-complexes to be finite or infinite. Note that she only works on finite CW-complexes, but in this paper, we study on some finite or infinite CW-complexes. Moreover, we compute the exact capacity of such spaces. Also, we concentrate on some questions of Borsuk which are stated in \cite{So} as follows:

Borsuk in \cite{So} stated some questions concerning with properties of the capacity of
compacta. In this paper, we give partial answer to three of these questions in some cases. The first question is the following one:
\begin{center}
 1. Is $C( X \times Y)$ determined by $C(X)$ and $C(Y)$?
\end{center}
In Section 4, we give a partial positive answer to this question as follows: If $X$ and $Y$ are Eilenberg-MacLane CW-complexes $K(G,n)$ and $K(H,m)$, respectively, such that $n\neq m$ and $G$ and $H$ are Hopfian groups, then $C(X\times Y)=C(X)\times C(Y)$  (see Proposition \ref{Q1}).

The second question is:
\begin{center}
2. Is $C(X \cup Y)$ determined by $C(X)$, $C(Y)$ and $C(X \cap Y)$?
\end{center}
Kolodziejczyk in \cite{16} gave  a negative answer to this question. She proved that there exist two finite CW-complexes $X$ and $Y$ with
$\dim X = \dim Y = 2$ such that $C(X), C(Y)$ and $C(X \cap Y)$ are finite, while
$C(X \cup Y)$ is infinite. In Section 3, we show that for Moore spaces $X=M(A,n)$ and $Y=M(B,m)$, $C(X\vee Y)=C(X)\times C(Y)$ provided that  $n\neq m$, $n,m \geq 2$, $A$ and $B$ are abelian Hopfian groups and $(X,x_0 )$ and $(Y,y_0)$ are good (see Proposition \ref{10001}). Recall that, a Moore space  is a simply connected $CW$-complex $X$ with a single non-vanishing homology group for some $n\geq 2$, that is $\tilde{H}_{i}(X,\mathbb{Z})=0$ for $i\neq n$.

The next question is as follows:
\begin{center}
3.  Is the capacity $C(A)$ determined by the homology properties of $A$?
\end{center}
In Section 3, we show that the answer to the above question is positive for Moore spaces. In fact, we prove that there is a one-to-one corresponding between homotopy classes of spaces homotopy dominated by $M(A,n)$ and direct summands of $A$ up to isomorphism, for $n\geq 2$ (see Proposition \ref{100}).

Borsuk in \cite{So} asked another question on the capacity of finite polyhdera as follows:
\begin{center}
 Is it true that the capacity of every finite polyhedron is finite?
\end{center}
Kolodziejczyk in \cite{16} showed that there exists a polyhedron (even of dimension 2) homotopy dominates infinitely many polyhedra of different homotopy types, and so she gave a negative answer to this question.
Moreover, she proved that such examples are not rare, for every non-abelian poly-$\mathbb{Z}$-group
$G$ and an integer $n\geq 3$ there exists a polyhedron $P$ with $\pi_1 (P)\cong G$ and $\dim P = n$
dominating infinitely many polyhedra of different homotopy types (see \cite{13}). In particular, there
exist polyhedra with nilpotent fundamental groups and infinite capacity.
Also, she gave positive answer to these questions under some conditions: in \cite{15} she proved (using the results of localization theory in the homotopy category of CW-complexes) that every simply connected polyhedron dominates only finitely many different homotopy types. In \cite{14} she also proved  that
polyhedra with finite fundamental groups dominate only finitely many different homotopy
types. In  \cite{18}, by extending the methods of \cite{14}, she proved that for some classes of polyhedra
with abelian fundamental groups, the answer to the above question is positive.
She also proved that every nilpotent polyhedron dominates only finitely many different
homotopy types.\\

In this paper, we compute the capacity of some well-known topological spaces exactly. We compute the exact capacity of Moore spaces $M(A,n)$ and Eilenberg-MacLane spaces $K(G,n)$ (in finite or infinite cases). In fact, we show that the capacity of a Moore space $M(A,n)$ and an Eilenberg-MacLane space $K(G,n)$ equals to the number of direct summands of $A$ and $G$, respectively, up to isomorphism. Also, we compute the capacity of the wedge sum of finitely many Moore spaces of different degrees and the capacity of the product of finitely many Eilenberg-MacLane spaces of different homotopy types. In particular, we compute the capacity of the wedge sum of finitely many spheres of the same or different dimensions. Note that Borsuk in \cite{So} has mentioned that $C(\mathbb{S}^n)=2$ and $C(\bigvee_k \mathbb{S}^1 )=k+1$.

W. Holsztynski in \cite{12} proved that the number of homotopy idempotents of a CW-complex is an  upper bound for the capacity of it. Finally,  we show that this upper bound is not so good (see Remark \ref{899}).

\section{Preliminaries}
In this paper every topological space is assumed to be connected. We expect that the reader is familiar with the basic notions and facts of shape theory (see \cite{B1} and \cite{Mar}) and retract theory \cite{ret}. We need the following results and definitions for the rest of the paper.
\begin{theorem}\label{3030}\cite{3}.
If a map $f :X\longrightarrow Y$ between connected CW complexes induces isomorphisms
$f_* :\pi_n (X)\longrightarrow \pi_n (Y)$  for all n, then f is a homotopy equivalence.
\end{theorem}
\begin{theorem}\label{3031}\cite{3}.
A map $f :X\longrightarrow Y$ between simply-connected CW complexes is a homotopy
equivalence if $f_* :H_n (X) \longrightarrow H_n (Y) $  is an isomorphism for each n.
\end{theorem}
\begin{theorem}\label{-1}\cite{4}.
1) A connected CW-space X is contractible if and only if all its homotopy groups $\pi_n (X)$ ($n\geq 1$) are trivial.

2) A simply connected CW-space X is contractible if and only if all its homology groups $H_n (X)$ ($n\geq 2$) are trivial.
\end{theorem}
\begin{definition}\cite{4}.
Let $\lambda : \mathcal{C} \longrightarrow \mathcal{D}$ be a functor. By the sufficiency and the
realizability conditions, with respect to $\lambda$, we mean the following:
\begin{enumerate}[(a)]
\item
\textbf{Sufficiency}: if $\lambda (f)$ is an isomorphism, then so is $f$, where $f$ is a morphism in
$\mathcal{C}$. That is, the functor $\lambda$ reflects isomorphisms.
\item
\textbf{Realizability}: two following conditions  satisfy:
\begin{itemize}
\item
The functor $\lambda$ is representative, that is, for each object $D$ in $\mathcal{D}$ there is an
object $C$ in $\mathcal{C}$ such that $\lambda (C)$ is isomorphic to $D$. In this case, we say that
$D$ is $\lambda$-realizable.
\item
The functor $\lambda$ is full, that is, for objects X, Y in $\mathcal{C}$ and for each
morphism $f: \lambda (X)\longrightarrow \lambda (Y)$ in $\mathcal{D}$ there is a morphism $f_0 : X \longrightarrow Y$ in $\mathcal{C}$ with $\lambda (f_0 ) =f$. In this case, we also say that $f$ is $\lambda$-realizable.
\end{itemize}
\end{enumerate}
\end{definition}
\begin{definition}\cite{4}.
We call $\lambda : \mathcal{C}\longrightarrow \mathcal{D}$  a detecting functor if $\lambda$ satisfies both the
sufficiency and the realizability conditions, or equivalently if $\lambda$ reflects isomorphisms,
is representative and full.
\end{definition}
A faithful detecting functor is called an equivalence of categories. By a faithful functor, we mean a functor $\lambda : \mathcal{C}\longrightarrow \mathcal{D}$ such that the induced maps $\lambda :Hom (X,Y)\longrightarrow Hom(\lambda X,\lambda Y)$ are injective, for all objects $X,Y \in \mathcal{C}$ (see \cite{4}).
\begin{lemma}\label{1-1}\cite{4}.
A detecting functor $\lambda : \mathcal{C} \longrightarrow \mathcal{D}$ induces a 1-1 correspondence
between equivalence classes of objects in $\mathcal{C}$ and equivalence classes of objects
in $\mathcal{D}$.
\end{lemma}
\begin{definition}\cite{18}.
A homomorphism $f :G \longrightarrow H$ of groups is an $r$-homomorphism if there
exists a homomorphism $g :H\longrightarrow G$ such that $f\circ g = id_H$. Then $H$ is an $r$-image of $G$.
\end{definition}
In particular, let $G$ be a group with a subgroup $H$. Then $H$ is called a retract of $G$ if there exists a homomorphism $r:G\longrightarrow H$ such that $r\circ i =id_H$ where $i:H\longrightarrow G$ is the inclusion homomorphism.
\begin{lemma}\label{409}
Every $r$-image of an arbitrary group $G$ is a semidirect factor of it and vice versa.
\end{lemma}
\begin{proof}
It can be concluded by the definition of semidirect product.
\end{proof}
\begin{corollary}\label{408}
Let $G$ be an abelian group. Then cardinality of the following three sets are equal.
\begin{enumerate}
\item
The set of r-images of $G$, up to isomorphism.
\item
The set of retracts of $G$, up to isomorphism.
\item
The set of direct summands of $G$, up to isomorphism.
\end{enumerate}
\end{corollary}
\begin{proof}
$(1)\: \& \:(3)$: this is a direct result of Lemma \ref{409}.

$(1) \: \& \: (2)$: by definition, any retract of $G$ is an r-image of it. Also, for any r-image $H$ of $G$,  there exist homomorphisms $f :G \longrightarrow H$  and $g :H\longrightarrow G$  such that $f\circ g = id_H$. It is easy to show that $g(H)\: \big( \cong H \big)$ is a retacrt of $G$.
\end{proof}

\begin{proposition}\label{101}
Let $G$ be a finitely generated abelian group with the following form:
\[
\mathbb{Z}_{p_{1}^{\alpha_1}}^{(k_{1})}\oplus \mathbb{Z}_{p_{2}^{\alpha_2}}^{(k_{2})} \oplus \cdots \oplus \mathbb{Z}_{p_{n}^{\alpha_n}}^{(k_{n})},
\]
where for $i\neq j$, $p_{i}^{\alpha_i}\neq p_{j}^{\alpha_j}$, $p_i$'s are prime numbers, $\alpha_i$'s are non-negative integers, $\mathbb{Z}_{p_{i}^{\alpha_i}}^{(k_{i})}$ is the direct sum of $k_i$ copies of $\mathbb{Z}_{p_{i}^{\alpha_i}}$, and $\mathbb{Z}_{1}=\mathbb{Z}$. Then the number of direct summands of $G$, up to isomorphism,  is equal to
\[
(k_{1}+1)\times \cdots \times (k_{n}+1).
\]
\end{proposition}
\begin{proof}
The proof is in three steps.

Step One. For each $1\leq i\leq n$, the number of direct summands of $\mathbb{Z}_{p_{i}^{\alpha_i}}^{(k_{i})}$, up to isomorphism, is equal to $k_i +1$.

For this, it is obvious that for every $0\leq t \leq k_i$, $\mathbb{Z}_{p_{i}^{\alpha_i}}^{(t)}$ is a direct summand of $\mathbb{Z}_{p_{i}^{\alpha_i}}^{(k_{i})}$ and for each $0 \leq t \neq t' \leq k_i$, we have $\mathbb{Z}_{p_{i}^{\alpha_i}}^{(t)} \not \cong \mathbb{Z}_{p_{i}^{\alpha_i}}^{(t')}$. Now, suppose that $C$ is a direct summand of $\mathbb{Z}_{p_{i}^{\alpha_i}}^{(k_{i})}$. There exists a subgroup $D$ of $\mathbb{Z}_{p_{i}^{\alpha_i}}^{(k_{i})}$ such that $\mathbb{Z}_{p_{i}^{\alpha_i}}^{(k_{i})} \cong C \oplus D$. By \cite[Corollary 2.1.7]{Hu},  $C$ is a finitely generated abelian group. Suppose that $C\cong \mathbb{Z}_{q_{1}^{\beta_1}}^{(l_{1})}\oplus \cdots \oplus \mathbb{Z}_{q_{s}^{\beta_s}}^{(l_{s})}$. Since $C$ is a direct summand of $\mathbb{Z}_{p_{i}^{\alpha_i}}^{(k_{i})}$ and for every $1\leq j \leq s$, $\mathbb{Z}_{q_{j}^{\beta_j}}^{(l_{j})}$ is a direct summand of $C$, so for every $1\leq j\leq s$, $\mathbb{Z}_{q_{j}^{\beta_j}}^{(l_{j})}$ is a direct summand of $\mathbb{Z}_{p_{i}^{\alpha_i}}^{(k_{i})}$. Now, by uniqueness of decomposition of finitely generated abelian groups \cite[Theorem 2.2.6, (iii)]{Hu}, for any $j=1,\cdots ,s$, there exists an $i=1,\cdots ,n$ such that  $q_j =p_i$ and $\beta_j =\alpha_i$. Hence, $C\cong \mathbb{Z}_{p_{i}^{\alpha_i}}^{(t)}$ for some $0\leq t\leq k_i$.

Step Two. The number of direct summands of $\mathbb{Z}_{p_{i}^{\alpha_i}}^{(k_{i})} \oplus \mathbb{Z}_{p_{j}^{\alpha_j}}^{(k_{j})}$ for $i\neq j$, up to isomorphism, is equal to $(k_i +1)(k_j +1)$.

It is easy to see that for every $0\leq t\leq k_i$ and $0\leq s\leq k_j$, $\mathbb{Z}_{p_{i}^{\alpha_i}}^{(t)}\oplus \mathbb{Z}_{p_{j}^{\alpha_j}}^{(s)}$ is a direct summand of $\mathbb{Z}_{p_{i}^{\alpha_i}}^{(k_{i})} \oplus \mathbb{Z}_{p_{j}^{\alpha_j}}^{(k_{j})}$. Now similar to Step One, suppose that $C$ is a direct summand of $\mathbb{Z}_{p_{i}^{\alpha_i}}^{(k_{i})} \oplus \mathbb{Z}_{p_{j}^{\alpha_j}}^{(k_{j})}$ and $D$ is a subgroup of $\mathbb{Z}_{p_{i}^{\alpha_i}}^{(k_{i})} \oplus \mathbb{Z}_{p_{j}^{\alpha_j}}^{(k_{j})}$ such that $\mathbb{Z}_{p_{i}^{\alpha_i}}^{(k_{i})} \oplus \mathbb{Z}_{p_{j}^{\alpha_j}}^{(k_{j})} \cong C\oplus D$. Suppose $C \cong\mathbb{Z}_{q_{1}^{\beta_1}}^{(l_{1})}\oplus \cdots \oplus \mathbb{Z}_{q_{s}^{\beta_s}}^{(l_{s})}$. Since for every $1\leq m\leq s$, $\mathbb{Z}_{q_{m}^{\beta_m}}^{(l_{m})}$ is a direct summand of $\mathbb{Z}_{p_{i}^{\alpha_i}}^{(k_{i})} \oplus \mathbb{Z}_{p_{j}^{\alpha_j}}^{(k_{j})}$, so similar to the above argument, $C\cong  \mathbb{Z}_{p_{i}^{\alpha_i}}^{(t)} \oplus  \mathbb{Z}_{p_{j}^{\alpha_j}}^{(s)}$ for some $0\leq t\leq k_i$ and $0\leq s \leq k_j$.

Step Three: the number of direct summands of $\mathbb{Z}_{p_{1}^{\alpha_1}}^{(k_{1})}\oplus \mathbb{Z}_{p_{2}^{\alpha_2}}^{(k_{2})} \oplus \cdots \oplus \mathbb{Z}_{p_{n}^{\alpha_n}}^{(k_{n})}$, up to isomorphism, is equal to $(k_1 +1)(k_2 +1) \cdots (k_n +1)$.

It is obvious by the induction on n and using Step Two.
\end{proof}
\section{The Capacity of Mooer Spaces}
In this section, we compute the capacity of Moore spaces exactly. Also, we give the exact capacity of wedge sum of finitely many Moore spaces of different degrees. In particular, we compute the capacity of wedge sum of finitely many spheres of the same or different dimensions.
\begin{definition}\citep{4}.
A Moore space of degree $n$ $(n\geq 2)$ is a simply connected $CW$-space X with a single non-vanishing homology group of degree $n$, that is $\tilde{H}_{i}(X,\mathbb{Z})=0$ for $i\neq n$. We write $X=M(A,n)$ where $A\cong \tilde{H}_{n} (X,\mathbb{Z})$.
\end{definition}
Note that for $n=1$, the Moore space $M(A,1)$ can not be defined, because of some problems in existence and uniqueness of the space (for more details see \cite{3}).

The $(n-1)$-fold suspension \cite{1} of a pseudo projective plane $\mathbb{P}_q = \mathbb{S}^1\cup_q e_2$ \cite{4}, is a Moore space of degree $n$, that is
\[
\Sigma^{n-1}\mathbb{P}_q =M(\mathbb{Z}_q ,n).
\]
Recall that $\mathbb{P}_q$ is the sapce obtained by attaching a 2-cell $e_2$ to $\mathbb{S}^1$ by a map $q: \mathbb{S}^1 \longrightarrow \mathbb{S}^1$ of degree $q$.

It is also obvious that the sphere $\mathbb{S}^n$ is also a Moore space, $\mathbb{S}^n = M(\mathbb{Z}, n)$.
\begin{theorem}\label{-2}\cite{4}.
The homotopy type of a CW complex Moore space $M(A,n)$ is uniquely determined by $A$ and $n$ ($n > 1$).
\end{theorem}
Let $\mathbf{Ab}$ be the category of abelian groups and for $n\geq 2$ let
$\mathbf{M}^n \subset hTop $ be the full subcategory of the category $hTop$ consisting of spaces $M(A, n)$
with $A \in \mathbf{Ab}$.
\begin{theorem}\label{0}\cite{4}.
For any $n\geq 2$, the functor $H_n :\mathbf{M}^n \longrightarrow \mathbf{Ab}$ is a detecting functor, that is, for each
abelian group $A$ there is a Moore space $M(A, n)$, whose homotopy type is well defined by $(A, n)$. Moreover, for each homomorphism $\phi : A \longrightarrow B$,
there is a map $\bar{\phi}: M(A, n) \longrightarrow M(B, n)$ with $H_n (\bar{\phi})= \phi$.
\end{theorem}
Note that homotopy class $\bar{\phi}$ in above theorem is not uniquely determined by $\phi$.

For n $(n\geq 2)$, $\mathbf{FM}^n$  is the full subcategory of $\mathbf{M}^n$ consists of all
Moore spaces $M(A, n)$ where $A$ is a finitely generated abelian group (see \cite{4}). For
each such group we have a direct sum decomposition
\[
\mathbb{Z}_{q_1}\oplus \mathbb{Z}_{q_2} \oplus \cdots \oplus \mathbb{Z}_{q_{r}}, \quad q_i \geq 0
\]
of cyclic groups. Associated with this isomorphism there is a homotopy
equivalence
\[
M(A, n) \simeq \Sigma^{n-1} \Big( \mathbb{P}_{q_{1}}\vee \mathbb{P}_{q_{2}} \vee \cdots \vee \mathbb{P}_{q_{r}}\Big),
\]
where $\mathbb{P}_{n}=\mathbb{S}^1 \cup_{n} e^2 $ is a pseudo-projective plane if $n > 0$, and  $\mathbb{P}_0 =\mathbb{S}^1$ (see \cite{4}).
\begin{lemma}\label{-5}
Let $A$ be an abelian group and $n\geq 2$. Then a space $X$ is  homotopy dominated by  Moore space $M(A,n)$ if and only if $X$ is of the same homotopy type as $M(B,n)$ where $B$ is a direct summand of $A$.
\end{lemma}
\begin{proof}
Suppose that $X$ is  homotopy dominated by $M(A,n)$. Then $H_i (X)$ is a  direct summand of $H_i (M(A,n))$, for each $i\geq 0$. Hence $H_i (X)=0$ for each $i\neq n$ and $H_n (X)=B$ where $B$ is a direct summand of $A$.  Therefore, $X$ is a Moore space of the form $M(B,n)$. Conversely, suppose that $f_0 :B\longrightarrow A$ and $g_0 :A \longrightarrow B$ are homomorphisms such that $g_0 \circ f_0 =id_{B}$. By Theorem \ref{0}, we have $H_n (M(A,n))=A$ and $H_n (M(B,n))=B$ and there exist maps $g:M(A,n)\longrightarrow M(B,n)$ and $f:M(B,n)\longrightarrow M(A,n)$ such that $H_n ([g])=g_0$ and $H_n ([f])=f_0$. We have $H_n ([g\circ f])=H_n ([g]\circ [f])=H_n ([g])\circ H_n ([f])=g_0 \circ f_0 =id_{B}$. On the other hand, $H_n ([id_{M(B,n)}])=id_B$. Now, by  sufficiency condition of $H_n$, $g\circ f$ is an equivalence on $M(B,n)$ and so by Lemma \ref{1-1}, we have $g\circ f\simeq id_{M(B,n)}$. Hence $M(B,n)$ is homotopy dominated by $M(A,n)$.
\end{proof}
The previous lemma and Theorem \ref{-2} imply the following result.
\begin{proposition}\label{100}
There is a one-to-one corresponding between spaces homotopy dominated by $M(A,n)$ up to homotpy equivalence and direct summands of $A$ up to isomorphism, for $n\geq 2$.
\end{proposition}
The following proposition is a consequence of  Lemma \ref{101} and Proposition \ref{100}.
\begin{proposition}\label{8000}
Let $X$ be a Moore Space $M(A,m)$ $(m\geq 2)$, where $A$ is a finitely generated abelian group of the form
\[
\mathbb{Z}_{p_{1}^{\alpha_1}}^{(k_{1})}\oplus \mathbb{Z}_{p_{2}^{\alpha_2}}^{(k_{2})} \oplus \cdots \oplus \mathbb{Z}_{p_{n}^{\alpha_n}}^{(k_{n})},
\]
where for $i\neq j$, $p_{i}^{\alpha_i}\neq p_{j}^{\alpha_j}$, $p_i$'s are prime numbers, $\alpha_i$'s are non-negative integers, $\mathbb{Z}_{p_{i}^{\alpha_i}}^{(k_{i})}$ is the direct product of $k_i$ copies of $\mathbb{Z}_{p_{i}^{\alpha_i}}$ and $\mathbb{Z}_{1}=\mathbb{Z}$. Then the capacity of $X$ is exactly
\[
(k_{1}+1)\times \cdots \times (k_{n}+1).
\]
\end{proposition}
As an example, by Proposition \ref{100} the capacity of the Moore space $M(\mathbb
Q,n)$ is exactly 2. Recall that $\mathbb{Q}$ is not the direct sum of any family of its proper subgroups.   Also, by Proposition \ref{8000} the capacity of $M(\mathbb{Z}_2 \oplus \mathbb{Z}_2 \oplus \mathbb{Z}_3 \oplus\mathbb{Z} \oplus\mathbb{Z},n)$, $M(\mathbb{Z}_9 \oplus \mathbb{Z}_{64})$, $M(\mathbb{Z},n)$ and $M(\mathbb{Z}_{p^m},n)$ are exactly 18, 4, 2 and 2, respectively.
\begin{remark}
The exact computation of capacity of the wedge sum of finitely many spheres with the same or  different dimesnions seems interesting. In \cite{16}, it has been mentioned that the capacity of $\bigvee_k \mathbb{S}^1$ equals to $k+1$, but the proof does not work for $n\geq 2$. Kolodziejczyk in \cite{18} asked the following question:

 Does every polyhedron P with the abelian fundamental group $\pi_1 (P)$ dominate only finitely many different homotopy types?

She proved that two extensive classes of polyhdera, polyhedra with finite fundamental group, and polyhedra $P$ with abelian fundamental groups and finitely generated homology groups $H_i (\tilde{P})$  $(i\geq 2)$,  have finite capacity, where $\tilde{P}$ is the universal covering of $P$ (see \cite{14},\cite{18}). The wedge sum $\mathbb{S}^1 \vee \mathbb{S}^2$ is a simple example of a polyhedron $P$ with abelian fundamental group $\pi_1 (P )$ and infinitely generated homology group $H_2 (\tilde{P}; \mathbb{Z})$ where finiteness of its capacity is still unknown. Note that $\mathbb{S}^1 \vee \mathbb{S}^2$ is neither a Moore space nor an Eilenberg-MacLane space.
\end{remark}
In continue, we compute the capacity of wedge sum of finitely many Moore spaces with the same or  different dimesnions such as $\bigvee_{i\in \{ 1,\cdots ,k\}} \mathbb{S}^n$ and $\mathbb{S}^m \vee \mathbb{S}^n$ $(m,n \geq 2 , m\neq n)$ and in general, $\bigvee_{n\in I} (\vee_{i_n} \mathbb{S}^n)$ where $\vee_{i_n} \mathbb{S}^n$ denotes the wedge sum of $i_n$ copies of $\mathbb{S}^n$. For this, we need the following theorem.
\begin{theorem}\label{7}\cite{3}.
For a wedge sum $\bigvee_{\alpha} X_{\alpha}$ , the inclusions $i_{\alpha} :X_{\alpha}\hookrightarrow \bigvee_{\alpha} X_{\alpha}$ induce an  isomorphism $\bigoplus_{\alpha} i_{\alpha^*}:\bigoplus_{\alpha} \tilde{H}_{n} (X_{\alpha })\longrightarrow \tilde{H}_n (\bigvee_{\alpha}X_{\alpha })$, provided that the wedge sum is formed at basepoints $x_{\alpha}\in X_{\alpha }$, such that the pairs $(X_{\alpha },x_{\alpha })$'s  are good.
\end{theorem}
By a good pair $(X,A)$, we mean a topological space $X$ and a nonempty closed subspace $A$ of $X$, where $A$ is also a deformation retract of $X$. For any CW-complex $X$ and any subcomplex $A$ of $X$, $(X,A)$ is a good pair.
\begin{corollary}\label{1010}
Let $n\geq 2$ be fixed. The capacity of $\bigvee_{\alpha \in I}\mathbb{S}^n $ is finite if and only if $I$ is finite. In particular, $C(\bigvee_{i\in \{ 1,\cdots ,k\}}\mathbb{S}^n)=k+1$, for every $n\geq 1$.
\end{corollary}
\begin{proof}
Suppose that the capacity of $\bigvee_{\alpha \in I}\mathbb{S}^n$ is finite. Since $H_n(\bigvee_{\alpha \in I}\mathbb{S}^n )=\bigoplus_{\alpha \in I} \mathbb{Z}$, by  Proposition \ref{100}, the set $I$ is finite. Conversely, suppose $I$ is finite and $| I |=k$. By Theorem \ref{7}, $H_{n}(\bigvee_{i\in \{ 1,\cdots ,k\}}\mathbb{S}^n)=\mathbb{Z}^{(k)}$ and since $\bigvee_{i\in \{ 1,\cdots ,k\}}\mathbb{S}^n$ is a Moore space of degree $n$, by Proposition \ref{8000}, the capacity of $\bigvee_{i\in \{ 1,\cdots ,k\}}\mathbb{S}^n$ is equal to $k+1$.
\end{proof}
Recall that a group $G$ is called Hopfian if every epimorphism $f :G\longrightarrow G$ is an automorphism
(equivalently, $N = 1$ is the only normal subgroup of $G$ for which $G/N\cong G$). It is easy to see that if
$G$ is a Hopfian group and $H\cong G$, then $H$ is also Hopfian. Moreover, if $G$ is an abelian Hopfian group and  $K$ is a direct summand of $G$, then $K$ is also Hopfian \cite{Rob}.
\begin{proposition}\label{1000}
Let $X=\bigvee_{\alpha \in I}M(A_{\alpha},n_{\alpha})$ where $n_{\alpha}$'s are distinct, $n_{\alpha}\geq 2$ and $A_{\alpha}$ is an abelian Hopfian group and the wedge sum $\bigvee_{\alpha \in I} M(A_\alpha ,n_{\alpha})$ is formed at basepoints $x_{\alpha}\in M(A_{\alpha} ,n_{\alpha})$ where the pairs $(M(A_{\alpha} ,n_{\alpha}),x_{\alpha} )$'s are good.
Then, every topological space homotopy dominated by  $X$ is of the same homotopy type  as $\bigvee_{\alpha \in I }M(B_{\alpha} ,n_{\alpha})$ where $B_{\alpha}$ is a direct summand of $A_{\alpha}$, for each $\alpha$.
\end{proposition}
\begin{proof}
Suppose that $Y$ is  homotopy dominated by $X$ with a domination map $g:X\longrightarrow Y$ and a converse map $f:Y\longrightarrow X$, i.e., $g\circ f\simeq id_{B}$. Put $H_{n_{\alpha}}(f) (H_{n_{\alpha}}(Y))=B_{\alpha}$, then $ B_{\alpha}$ is a direct summand of $H_{n_\alpha } (X)$. Also, since $n_{\beta}$'s are distinct, we have $H_{n_\alpha } (X)=H_{n_\alpha } (\bigvee_{\beta \in I}M(A_{\beta},n_{\beta}))=\bigoplus_{\beta \in I} H_{n_\alpha}(M(A_{\beta},n_{\beta}))=A_{\alpha}$. By Lemma \ref{-5}, $M(B_{\alpha} ,n_{\alpha})$ is homotopy dominated by $M(A_{\alpha} ,n_{\alpha})$, with a domination map  $d_{\alpha}:M(A_{\alpha} ,n_{\alpha}) \longrightarrow M(B_{\alpha} ,n_{\alpha}) $ and a converse map
$u_{\alpha}:M(B_{\alpha} ,n_{\alpha}) \longrightarrow M(A_{\alpha} ,n_{\alpha})$ for $\alpha \in I$. So, $\bigvee_{\alpha} M(B_{\alpha} ,n_{\alpha})$ is homotopy dominated by $X$ with a domination map $d=\bigvee_{\alpha \in I}d_{\alpha}:X \longrightarrow \bigvee_{\alpha} M(B_{\alpha} ,n_{\alpha})$ and a converse map $u=\bigvee_{\alpha \in I}u_{\alpha}:\bigvee_{\alpha} M(B_{\alpha} ,n_{\alpha})\longrightarrow X$ (note that the wedge sum is coproduct in the category of $hTop^*$).  Now, the map
\[
g\circ u:\bigvee_{\alpha} M(B_{\alpha} ,n_{\alpha}) \longrightarrow Y
\]
 is a homology equivalence between simply connected CW-complexes. For this, $(g\circ u)_*  (H_{n_\alpha}(\bigvee_{\alpha} M(B_{\alpha} ,n_{\alpha}))=(g\circ u)_* (B_{\alpha})=g_* (u_* (f_* (H_{n_{\alpha}} (Y))))=g_* (f_* (H_{n_{\alpha}} (Y)))=H_{n_{\alpha}} (Y)$. Now, $(g\circ u)_* $ is an epimorphism between two isomorphic Hopfian groups $B_{\alpha}$ and $H_{n_{\alpha}} (Y)	 $, which implies that $(g\circ u)_*$ is isomorphism and so, by Theorem \ref{3031}, $Y$ and $\bigvee_{\alpha} M(B_{\alpha} ,n_{\alpha})$ have the same homotopy type.
\end{proof}
The following corollary is a consequence of Proposition \ref{1000}.
\begin{corollary}\label{10001}
Let  $X=\bigvee_{n\in I} M(A_n ,n)$ where  $A_{\alpha}$'s are abelian Hopfian groups, $I$ is a subset of $\mathbb{N}\setminus \{ 1\}$ and the wedge sum is formed at basepoints $x \in M(A_n ,n)$ such that the pairs $(M(A_n ,n),x)$ are good. Then $C(X)=\prod_{n\in I} C(M(A_n ,n))$.
\end{corollary}
\begin{remark}
 Note that we can not omit the distinctness  condition of $n_{\alpha}$'s in Proposition \ref{1000}. If $X=\bigvee_{\alpha\in I}M(A_{\alpha},n)$  for a fixed natural number $n\geq 2$, then we need the structure of direct summands of $\bigoplus_{\alpha \in I}A_{\alpha}$ which is unknown in general.
\end{remark}
Now, we are in a position to compute the capacity of  wedge sum of finitely many spheres of the same or different dimensions.
\begin{corollary}
The capacity of $\bigvee_{n\in I} (\vee_{i_n} \mathbb{S}^n)$ is exactly $\prod_{n\in I}(i_n +1)$ where $\vee_{i_n} \mathbb{S}^n$ denotes the wedge sum of $i_n$ copies of $\mathbb{S}^n$, $I$ is a finite subset of natural numbers and $i_n \in \mathbb{N}$.
\end{corollary}
\begin{proof}
It can be concluded from Corollaries  \ref{1010} and \ref{10001}.
\end{proof}
\begin{example}
The capacity of $\mathbb{S}^m \vee \mathbb{S}^n$ $(m,n\geq 2, m\neq n)$ is exactly 4.
\end{example}
\section{The Capacity of Eilenberg-MacLane Spaces }
In this section, we intend to compute the capacity of Eilenberg-MacLane spaces exactly. Note that some of the results on the capacity of Eilenberg-MacLane spaces are similar to the results of previous section for Moore spaces, but their proofs are completely different. Also note that there are many spaces which are Moore spaces but they are not Eilenberg-MacLane spaces, and vise versa.  For example,  spheres $\mathbb{S}^n$ $(n\geq 2)$ are examples of  Moore spaces  which are not Eilenberg-MacLane spaces and $\mathbb{T}^k$ $(k\geq 1)$, k-dimensional torus, are Eilenberg-MacLane spaces which are not Moore spaces.

Recall that a space $X$ having just one nontrivial homotopy group $\pi_n (X)\cong G$ is called an Eilenberg-MacLane space and is denoted by $K(G,n)$. The full subcategory of the category $hTop$ consisting of spaces $K(G, n)$ with $G \in \mathbf{Gp}$ is denoted by $\mathbf{K}^n$ (see \cite{3}).
\begin{theorem}\label{asl}\cite{4}.
The $n$-th homotopy group functor $\pi_n : \mathbf{K}^n \longrightarrow \mathbf{Ab}$
is an equivalence of categories for $n\geq 2$. Moreover, the functor $\pi_1 :\mathbf{K}^1 \longrightarrow \mathbf{Gp}$ is also an equivalence of categoreis.
\end{theorem}

\begin{theorem}\label{10}\cite{3}.
The homotopy type of a CW complex $K(G,n)$ is uniquely determined by $G$ and $n$.
\end{theorem}
\begin{lemma}\label{13}
Let $G$ be a group. Then the space $X$ is homotopy dominated by the Eilenberg-MacLane space  $K(G,n)$ if and only if $X$ is of the same homotopy type as $K(H,n)$ where $H$ is an r-image of $G$.
\end{lemma}
\begin{proof}
Suppose that $X$ is  homotopy dominated by $K(G,n)$. Then  $\pi_i (X)$ is an  r-image of $\pi_i (K(G,n))$, for each $i\geq 1$. Hence $\pi_i (X)=0$ for each $i\neq n$ and $\pi_n (X)=H$ where $H$ is an r-image  of $G$.  Therefore $X$ is an Eilenberg-MacLane space of the form $K(H,n)$. Conversely, suppose that $\bar{f}:H\longrightarrow G$ and $\bar{g}:G\longrightarrow H$ are homomorphisms such that $\bar{g}\circ \bar{f}=id_H$. By Theorem \ref{asl}, there exist homotopy classes $f:K(H,n)\longrightarrow K(G,n)$ and $g:K(G,n)\longrightarrow K(H,n)$ such that $\pi_n ([f])=\bar{f}$ and $\pi_n ([g])=\bar{g}$. Also, since $\bar{g}\circ \bar{f}=id_H$,  we must have $g\circ f\simeq id_{K(H,n)} $. Hence $K(H,n)$ is homotopy dominated by $K(G,n)$.
\end{proof}
Now, similar to the Moore spaces, we have the following result on the capacity of Eilenberg-MacLane spaces.
\begin{proposition}\label{2000}
There exists a one-to-one correspondence between the set of all homotopy types of spaces homotopy dominated by the Eilenberg-MacLane space $K(G,n)$ and the set of  all isomorphism classes of  r-images of $G$.
\end{proposition}
\begin{proof}
By Lemma \ref{13}, every space homotopy dominated by $K(G,n)$ has the form $K(H,n)$, where $H$ is an r-image of $G$. Also, if $H$ is an r-image of G, then $K(H,n)$ is homotopy dominated by $K(G,n)$. Now, By Theorem \ref{10}, it is obvious that  $H\longmapsto K(H,n)$ is a one-to-one  correspondence between the set of all  r-images $H$ of $G$ and the set of all homotopy types of spaces homotopy dominated by $K(G,n)$.
\end{proof}
Note that by a result of Kolodziejczyk \cite{15}, the capacity of $K(G,n)$ is finite, for $n\geq 2$. Also, when $G$ is abelian, by another result of Kolodziejczyk \cite[Theorem 2]{18}, the capacity of $K(G,1)$ is also
 finite. Using  Corollary \ref{408} and Proposition \ref{2000}, we have the following corollary.
\begin{corollary}\label{20}
Let $G$ be an abelian group. Then the capacity of $K(G,n)$ $(n\geq 1)$ is finite if and only if $G$ has finitely many direct summands up to isomorphism.
\end{corollary}
In the following, we compute the capacity of $K(G,n)$ exactly, when $G$ is a finitely generated abelian group.
\begin{proposition}
Let $X$ be an Eilenberg-MacLane space $K(G,n)$ $(n\geq 1)$, where $G$ is a finitely generated abelian group of the form
\[
\mathbb{Z}_{p_{1}^{\alpha_1}}^{(k_{1})}\oplus \mathbb{Z}_{p_{2}^{\alpha_2}}^{(k_{2})} \oplus \cdots \oplus \mathbb{Z}_{p_{n}^{\alpha_n}}^{(k_{n})},
\]
where for $i\neq j$, $p_{i}^{\alpha_i}\neq p_{j}^{\alpha_j}$, $p_i$'s are prime numbers, $\alpha_i$'s are non-negative integers, $\mathbb{Z}_{p_{i}^{\alpha_i}}^{(k_{i})}$ is the direct sum of $k_i$ copies of $\mathbb{Z}_{p_{i}^{\alpha_i}}$ and $\mathbb{Z}_{1}=\mathbb{Z}$. Then the capacity of $X$ is exactly
\[
(k_{1}+1)\times \cdots \times (k_{n}+1).
\]
\end{proposition}
\begin{proof}
This is a consequence of Lemma \ref{101} and Proposition \ref{2000}.
\end{proof}
As an example, the capacity of $n$-dimensional torus $\mathbb{T}^n$ is exactly $n+1$. Note that $\mathbb{T}^n$ is not a Moore space, so its capacity can not be computed by the resulats of the previous section.
\begin{example}
$K(\mathbb{Q},1)$ is an infinite CW-complex of capacity 2. Indeed, $\mathbb{Q}$ is not finitely generated abelian group and so by \cite[Corollary 7.37]{1}, $K(\mathbb{Q},1)$ is an infinite CW-complex. Also, by Corollary \ref{20} and the fact that $\mathbb{Q}$ has only 2 r-images, the capacity of $K(\mathbb{Q},1)$ is 2.
\end{example}
The next corollary computes the capacity $K(G,n)$ for infinitely generated abelian group $G$ with some conditions.
\begin{corollary}
The capacity of $K(G,n)$ for finite rank torsion free abelian group $G$ is finite.
\end{corollary}
\begin{proof}
It can be concluded from  Corollary \ref{20} and the fact that $G$  has only finitely many direct summands, up to
isomorphism (see \cite{5}).
\end{proof}
\begin{remark}
By the definition of AKS $\mathbb{Z}$-module (for more details see \cite{6}), an abelian group is  $AKS$  $\mathbb{Z}$-module if and only if it has finitely many direct summands. Hence, we can rewrite Corollary \ref{20} for any abelian group $G$ as follows:
\begin{center}
 `` $K(G,n)$ has finite capacity if and only if $G$ is an AKS   $\mathbb{Z}$-module''
\end{center}
As an example, Artinian $\mathbb{Z}$-modules satisfy in the definition of AKS $\mathbb{Z}$-module.
\end{remark}
To compute  the capacity of finite product of Eilenberg-MacLane spaces, we give  the following proposition.
\begin{proposition}\label{120120}
Let  $X=\prod_{\alpha \in I} K(G_\alpha ,n_{\alpha})$ where $n_{\alpha}$'s are distinct, $n_{\alpha}\geq 1$ and $G_{\alpha}$'s  are Hopfian groups. Then every topological space homotopy dominated by  $X$ is of the form $\prod_{\alpha \in I} K(H_{\alpha} ,n_{\alpha})$ where $H_{\alpha}$ is an r-image  of $G_{\alpha}$, for each $\alpha$.
\end{proposition}
\begin{proof}
Suppose that the space $Y$ is  homotopy dominated by $X$ with a domination map $g:X\longrightarrow Y$ and a converse map $f:Y\longrightarrow X$. If we put $\pi_{n_\alpha}(f)(\pi_{n_{\alpha}} (Y))=H_{\alpha}$, then $ H_{\alpha}$ is an r-image of $\pi_{n_\alpha } (X)$. Also, since $n_\beta$'s are distinct, we have $\pi_{n_\alpha } (X)=\pi_{n_\alpha } (\prod_{\beta \in I} K(G_\beta ,n_{\beta}))=\prod_{\beta \in I} \pi_{n_\alpha}(K(G_\beta ,n_{\beta}))=G_{\alpha}$. By Lemma \ref{13}, $K(H_{\alpha}  ,n_{\alpha})$ is homotopy dominated by $K(G_{\alpha} ,n_{\alpha})$. Suppose that for any $\alpha$,  $d_{\alpha}:K(G_{\alpha} ,n_{\alpha}) \longrightarrow K(H_{\alpha}  ,n_{\alpha}) $ and
$u_{\alpha}:K(H_{\alpha}  ,n_{\alpha}) \longrightarrow K(G_{\alpha} ,n_{\alpha})$ is a domination and a converse map, respectively. Then $\prod K(H_{\alpha} ,n_{\alpha})$ is homotopy dominated by $X$ with  a domination map $d=\prod_{\alpha \in I}d_{\alpha}:X \longrightarrow \prod K(H_{\alpha} ,n_{\alpha})$ and a converse map $u=\prod_{\alpha \in I}u_{\alpha}:\prod K(H_{\alpha} ,n_{\alpha})\longrightarrow X$.  Now, the map
\[
g\circ u:\prod K(H_{\alpha}  ,n_{\alpha}) \longrightarrow Y
\]
 is a homotopy equivalence between connected CW-complexes since \\ $(g\circ u)_*  (\pi_{n_\alpha}(\prod K(H_{\alpha}  ,n_{\alpha}))=(g\circ u)_* (H_{\alpha})=g_* (u_* (f_* (\pi_{n_{\alpha}} (Y))))=g_* (f_* (\pi_{n_{\alpha}} (Y)))=\pi_{n_{\alpha}} (Y)$. Now, $(g\circ u)_*$ is an epimorphism between two isomorphic  Hopfian groups  $H_{\alpha}$ and $\pi_{n_{\alpha}} (Y)$ which implies that $(g\circ u)_*$ is isomorphism and so, by Theorem \ref{3030}, $Y$ and  $\prod_{\alpha \in I} K(f_* (H_{\alpha}) ,n_{\alpha})$ have the same homotopy type.
\end{proof}
\begin{corollary}\label{Q1}
Let $\{ K(G_n ,n)\}_{n \in I}$ be a family of Eilenberg-MacLane spaces where  $I$ is a subset of $\mathbb{N}$ and $G_{\alpha}$'s  are Hopfian groups.   Then
$C(\prod_{n \in I} K(G_n ,n) )=\prod_{n\in I} C\big( K(G_n,n)\big) $.
\end{corollary}
\begin{proof}
This is a direct result of Proposition \ref{120120}.
\end{proof}
\begin{remark}
Note that we can not omit the distinctness condition of $n_{\alpha}$'s in Proposition \ref{120120} . If $X=K(G_1 , n)\times K(G_2 ,n)$ for a fixed natural number $n$, then we need the structure of r-images of $G_1\times G_2$ which is unknown in general.
\end{remark}
Let $X$ be a topological space. The set of all maps $f:X\longrightarrow X$ satisfying the condition $f^2 =f$, constitute a subset of $X^X$ which is denoted by $\mathcal{R}(X)$ (see \cite{Concer}).
Also, the set of all homotopy classes of maps $f:X\longrightarrow X$ with $f^2 \simeq f$ which are called   homotopy idempotents of X, is denoted by $hI(X)$. Similarly, for a group $G$, the set of all homomorphisms $f:G\longrightarrow G$ with $f^2 =f$, is denoted by $\mathcal{R}(G)$.
\begin{remark}\label{899}
There is an upper bound for the capacity of any topological space $X$ as $C(X)\leq |hI(X)|$ (see \cite{12}). Here we show that $|hI(X)|$ is not a good upper bound for the capacity of $X$. For this, let $X$ be an Eilenberg-MacLane space $K(G,1)$. By Theorem \ref{asl}, the correspondence $f\longmapsto f_* $ induces a one-to-one correspondence between $[X,X]$ and $Hom(\pi_1 (X),\pi_1 (X))$. So the number of homotopy classes of maps $f:X\longrightarrow X$ with $f^2 \simeq f$ equals to the number of homomorphisms $g:\pi_{1}(X)\longrightarrow \pi_{1}(X)$ with $g^2 =g$. Hence $|hI(X)|=|\mathcal{R}(\pi_{1}(X))|$. Now,  suppose $X$ is the torus $\mathbb{T}^2$. Then $|hI(\mathbb{T}^2)|=|\mathcal{R}(\pi_{1}(\mathbb{T}^2))|$.   Since  $\pi_{1}(\mathbb{T}^2) \cong \mathbb{Z}\times \mathbb{Z}$, we have $Hom(\pi_1 (\mathbb{T}^2 ),\pi_1 (\mathbb{T}^2 ) )\cong M_{2}(\mathbb{Z})$. Therefore $|\mathcal{R}(\pi_1 (\mathbb{T}^2))|$  equals to the number of idempotent matrices in $M_{2}(\mathbb{Z})$. But $M_{2}(\mathbb{Z})$ has infinite number idempotents like $
\begin{pmatrix}
1 & n \\
0 & 0 \\
\end{pmatrix}
$
 for $n\in \mathbb{Z}$. Hence $hI(\mathbb{T}^2 )$ is infinite, while   $C(\mathbb{T}^2 )=3$.
\end{remark}













\end{document}